\documentclass[12pt]{article}
\usepackage{amssymb,amsthm,amsmath,amsfonts}
\usepackage{graphicx}
\usepackage{cite}
\usepackage{color}

\theoremstyle{definition}
\newtheorem{theorem}{Theorem}

\newtheorem{corollary}[theorem]{Corollary}
\newtheorem{proposition}[theorem]{Proposition}
\newtheorem{definition}[theorem]{Definition}

\newtheorem{notation}[theorem]{Notation}
\newtheorem{remark}[theorem]{Remark}
 
\newcommand{\bH}{{\mathbb{H}}}

\newcommand{\bR}{{\mathbb{R}}}

\DeclareMathOperator{\spn}{span}

\newcommand{\cA}{{\mathcal{A}}}

\begin{document}
 \title{Covariance property for a quaternionic fractional Cauchy-Riemann operator under affine transformation}
\small{
\author
{Isidro Paulino-Basurto$^{(1)}$, Jos\'e Oscar Gonz\'alez-Cervantes$^{(1)}$,\\ Juan Bory-Reyes$^{(2)}$, Hung Manh Nguyen$^{(3)\footnote{corresponding author}}$}
\vskip 1truecm
\date{\small $^{(1)}$ Departamento de Matem\'aticas, ESFM-Instituto Polit\'ecnico Nacional. 07338, Ciudad de M\'exico, M\'exico\\ Email: jogc200678@gmail.com, isidroo-@hotmail.com\\$^{(2)}$ {SEPI, ESIME-Zacatenco-Instituto Polit\'ecnico Nacional. 07338, Ciudad de M\'exico, M\'exico}\\Email: juanboryreyes@yahoo.com\\$^{(3)}$ {University of Transport and Communications, Hanoi, Vietnam}\\Email: hung.manh.nguyen@utc.edu.vn}}

\maketitle

\begin{abstract}
In this paper, we continue the development of the foundations of a quaternionic function theory that is associated to a fractional proportional Cauchy-Riemann type operator acting on quaternionic functions. We introduce a new class of fractional hyperholomorphic functions adapted to affine transformations of the domain, and study the corresponding quaternionic right module structure. In particular, we establish versions of the Stokes, Borel-Pompeiu and Cauchy formulas for this function theory. A covariance property of the fractional operator under affine transformations is proved, showing that the composition with an affine map, weighted by a suitable quaternionic factor, preserves the class of fractional hyperholomorphic functions. The result is shown to include the case considered previously by the authors, and extends their covariance property to the setting of affine transformations.
\end{abstract}


\section{Introduction}

In classical complex analysis, conformal mappings --- in particular
M\"{o}bius transformations --- play a fundamental role in both theory
and applications: they allow one to transform geometrically complex
domains into simpler ones while preserving the class of holomorphic
functions. This technique has found remarkable applications in fluid
dynamics, aerodynamics, and the analysis of stress fields near crack
tips in elastic materials~\cite{LN, LWP}. 
A crucial difference arises in quaternionic analysis: only M\"{o}bius transformations are conformal mappings in~$\mathbb{R}^{3}$ but they are
\emph{not} monogenic functions --- a generalization of holomorphic functions in higher dimensions. Moreover, the composition of a
monogenic function with a M\"{o}bius transformation does
not remain monogenic. This raises a fundamental question:
what structural equation does $f\circ T$ satisfy, and can one still
exploit M\"{o}bius transformations for domain transformation in
quaternionic function theory?

The covariance property of the $\psi$-Cauchy-Riemann generalized type operator, denoted by ${}^{\psi}\overline{\partial}$ is a natural extension of the chain rule, well-known in complex analysis. More precisely, it was shown in~\cite{GN1} that for any M\"{o}bius transformation~$T$ and any monogenic function~$f$, there always exists a structural set~$\psi'$, determined by~$T$, such that
$f\circ T$ is $\psi'$-hyperholomorphic. Equivalently, it states that for any M\"obius transformation $T$ {there} exist functions $\alpha$ and $\beta$, given in terms of $T$, such that  
$${}^{\psi}\overline{\partial }[\alpha f \circ T] = \beta \, {}^{\psi}\overline{\partial }[ f] \circ T,$$
for all functions $f$ of class $C^1$ in a domain in which the composition $f \circ T$ makes sense.   

This property of ${}^{\psi}\overline{\partial}$ induces an isomorphism  between the quaternionic modules $\textrm{ker} ({}^{\psi}\overline{\partial })$ associated to two conformally equivalent domains and this fact can be used to show the covariance invariant property of some quaternionic function spaces, such as Bergman, Dirichlet and Besov spaces induced by ${}^{\psi}\overline{\partial}$ (see \cite{GN2,GNL,S, SV1, SV2}).

In \cite{GB-1, GB-2, GPBN} the function theory induced by a family of fractional quaternionic Fueter,  Moisil-Teodorescu and $\psi$-Cauchy-Riemann generalized type operators has been established. In particular, the function  theory induced by a fractional proportional $\psi$-Cauchy-Riemann generalized type operator and some important  results such as Borel-Pompeiu, Stokes and Cauchy formulas were presented in \cite{GPBN}.

The present work introduces the  quaternionic function theory induced by a quaternionic fractional proportional $\psi$-Cauchy-Riemann generalized type operator under an affine transformation and proves its covariance property.

A natural and important question is whether the covariance
property of~${}^{\psi}\overline{\partial}$ extends to the fractional
proportional setting, where ${}^{\psi}\overline{\partial}$ is replaced
by the fractional proportional operator
${}_a^{\psi}\overline{\partial}^{\,\alpha,\sigma,\varphi}$
introduced in~\cite{GPBN}. The present work gives an
affirmative answer for affine transformations, which constitute an
important subclass of M\"{o}bius transformations, and establishes the
corresponding quaternionic function theory. This extends the classical
technique of domain transformation to the fractional proportional
setting, with potential applications to boundary value problems in
fractional quaternionic analysis.

The plan of the paper is as follows: Section \ref{preli} provides a brief introduction to some  basic concepts of fractional calculus, rudiments of quaternionic analysis and to the covariance property of the $\psi$-Cauchy-Riemann operator. {Concluding with the introduction of the function theory associated  with a  quaternionic fractional proportional operator. Meanwhile, Section \ref{MR} presents the  function theory induced by a  quaternionic fractional proportional operator associated to a Möbius transformation, an affine covariance property of the operator ${}_a^{\psi}\overline{\partial} ^{\alpha , \sigma , \varphi, T}$ and the behavior of the operator ${}_a^{\psi}\overline{\partial} ^{\alpha , \sigma , \varphi}$  under composition} with a conformal mapping.

\section{Preliminaries} \label{preli}

\subsection{Proportional fractional derivatives with respect to functions}

This subsection presents a brief summary of the basic concepts of some fractional proportional operators given in  \cite{JAAl, JAA, JUAB}.

A type of  proportional derivative  and integral are given  by 
\begin{align*} (D^{\rho}f )(t) = & (1-\rho)f (t) +  \rho f'(t), \\
({}_aI^{1-\rho} f) (t) = &\frac{1}{\rho}\int_
a^t  e^{ \frac{\rho-1}{\rho}(t-s)}f (s) ds, \quad \forall  f\in C^1(\mathbb R, \mathbb R), \nonumber  
 \end{align*}
where $\rho\in (0,1)$, respectively. In addition, the previous    proportional derivative of $f\in C^1(\mathbb R,\mathbb R)$ with respect to a function $\varphi \in C^1(\mathbb R, \mathbb R^+ )$, where $\mathbb R^+$ is the set of all real positive numbers,  is defined by  
\begin{align*} 
(D^{\rho,\varphi} f )(t) = (1-\rho)f(t) + \rho\frac{f'(t)}{\varphi'(t)}.
\end{align*}
On the other hand, given  $\alpha\in \mathbb C$ such that $0< \Re \alpha <1 $, then  the left and the right fractional  proportional integrals  with order $\alpha$  and  with respect to $\varphi$  are defined by  
\begin{align*}  
({}_a I^{\alpha,\rho,\varphi} f) (t) := &\frac{1}{\rho^\alpha \Gamma(\alpha)} \int_a^te^{\frac{\rho-1}{\rho} (\varphi(t)-\varphi(\tau)) }
(\varphi(t)- \varphi(\tau))
^{\alpha-1}f (\tau)\varphi'(\tau) d\tau , \nonumber \\
(I_b^{\alpha,\rho,\varphi} f)(t) :=& \frac{1}{\rho^\alpha \Gamma(\alpha)}
 \int_ t^b
e^{\frac{\rho-1}{\rho} (\varphi(\tau)- \varphi(t)) }
(\varphi (\tau)-  \varphi (t))
^{\alpha-1}f (\tau)\varphi'(\tau) d\tau, 
\end{align*}
respectively. Meanwhile, the left and the right fractional proportional derivatives with order $\alpha$ and with respect to $\varphi$  are
\begin{align}\label{DerFracPropRespecFuntL}
({}_aD^{\alpha, \rho,\varphi} f)(t) :=  D^{\rho,\varphi}
{}_aI^{1-\alpha,\rho,\varphi}f(t) \ \ \textrm{ and } \ \  
(D_b^{\alpha, \rho,\varphi} f)(t) := D^{\rho,\varphi}
 I_b^{1-\alpha,\rho,\varphi}f(t),
\end{align}
respectively.  
The fundamental theorems of these families of fractional differential  and  integrals operators are
\begin{align}\label{FundTheoFPW}
{}_aD^{\alpha,\rho,\varphi} \circ 
{}_aI^{\alpha,\rho,\varphi} f(t) = f(t) \quad \textrm{and} \quad
 D_b^{\alpha,\rho,\varphi} \circ 
 I_b^{\alpha,\rho,\varphi} f(t) = f(t).
\end{align}
If $\varphi$ is the identity function    and $\rho=1$ we see that  
\begin{align*}
({}_aD^{\alpha, 1,I} f)(t)  = &   \frac{\partial}{\partial t}{}_aI^{1-\alpha}f(t)=:
({}_aD^{\alpha} f)(t), \\
 (D_b^{\alpha, 1,I} f)(t)=  & (D_b^{\alpha} f)(t), 
\end{align*}
are the well-known left- and right- fractional derivative of $f$ of order $\alpha$  in the Riemann-Liouville sense.

\subsection{On $\psi$-Cauchy-Riemann operator calculus in $\mathbb R^3$}

A summary of basic  rudiments of quaternionic analysis and  the conformable property of the $\psi$-Cauchy-Riemann operators are presented in  this subsection.
 \\
  The skew-field of real quaternions,  $\mathbb H$, has the   standard orthonormal basis   $\{{\bf e_0}, {\bf e_1}, {\bf e_2}, {\bf e_3}\}$, where ${\bf e_0}$ is the multiplicative identity of $\bH$ and  the set of imaginary units $\{{\bf e_1}, {\bf e_2}, {\bf e_3}\}$ satisfy
$$ {\bf e_1}^{2}= {\bf e_2}^{2}= {\bf e_3}^{2}=-{\bf e_0},$$
$$ {\bf e_1}\, {\bf e_2}=- {\bf e_2}\,\  {\bf e_1}= {\bf e_3};\   {\bf e_2}\, {\bf e_3}=- {\bf e_3}\, {\bf e_2}= {\bf e_1};\  {\bf e_3}\, {\bf e_1}=- {\bf e_1}\, {\bf e_3}= {\bf e_2}.$$ 
  $ \mathbb H$ consists of elements $q=x_0 {\bf e_0}+x_{1} {\bf e_1}+x_{2} {\bf e_2}+x_{3} {\bf e_3}$,   
where $x_{k}\in \mathbb R$  for  $k= 0,1,2,3$. The sum is developed between coefficients of basic elements and the  non-commutative multiplication obeys  rules of $\{{\bf e_1}, {\bf e_2}, {\bf e_3}\}$.
\\
The real span of $\{\bf e_0, \bf e_1, \bf e_2\}$ is called the set of reduced quaternions and  forms a real subspace of $\mathbb H$ but not a sub algebra.
\\
The quaternionic conjugation  in $\mathbb H$  is given by 
$$q\rightarrow {\overline q}:=x_0 {\bf e_0}-x_{1}{\bf e_1}-x_{2}{\bf e_2}-x_{3}{\bf e_3}$$ 
and the norm is 
$$|q|  := \sqrt{ x^{2}_{0}+x^{2}_{1}+x^{2}_{2}+x^{2}_{3}}= \sqrt{ q\,{\overline q}}=\sqrt{{\overline q}\,q}.$$
The scalar product in $\mathbb H$ is given by 
$$\langle u, v\rangle:=\frac{1}{2}(\bar u v + \bar v u) \qquad \text{for all } u,v\in\mathbb{H}.$$
An ordered set $\psi:=\{\psi_0, \psi_1,\psi_2\}\subset \spn_{\mathbb R}\{\bf e_0, \bf e_1, \bf e_2\} $ is called a structural set if it satisfies the orthonormality condition:  
$$\psi_k{\overline \psi}_s+\psi_s{\overline\psi}_k=\langle \psi_k, \psi_s\rangle =2\delta_{k,s},$$ 
for $k,s=0,1,2$, where $\delta_{k,s}$ is the Kronecker symbol. 
Since $\psi_0,\psi_1,\psi_2$ have no $\mathbf{e_3}$-component, the unit vector $\mathbf{e_3}$ in
$\mathbb{H}$ is orthogonal to all of them. 

We consider $\psi_3$ a unit vector such that $\{ \psi_0,\psi_1,\psi_2,\psi_3\}$ is a orthonormal basis of $\mathbb{H}$ and $ \{\psi_1,\psi_2,\psi_3\} $ is cooriented with $\{e_1,e_2,e_3\}$.To simplify the notation, $\psi_0=1$  is usually used.
\\
Denote
  $$\mathcal A:=\spn_{\mathbb R}\{\bf \psi_0, \psi_1, \psi_2\}  \cong\mathbb R^3  .$$
The real three-dimensional vector space $\mathbb R^3$ will be embedded in $\mathbb H$ using  the mapping  $x= (x_0,x_1,x_2)\in \mathbb R^3 \mapsto x_0\psi_0 + x_1\psi_1+ x_2 \psi_2\in \mathcal A$.   
\\
Let $\Omega\subset \mathcal A$ be an open bounded domain whose  boundary $\partial \Omega$  is a smooth surface.  All $\mathbb H$-valued functions $f$ defined on $\Omega$ are represented  by  $f=\sum_{k=0}^3 f_k \psi_k$, where $f_k:\Omega\to \mathbb R$  for $ k= 0,1,2,3$. The function $f$ is said to be continuously differentiable (denote $f \in C^{1}(\Omega, \mathbb H)$) if every  $f_k$ has the same property.
 
 \begin{definition}
The left and the right-$\psi$-Cauchy-Riemann operators acting on $f \in C^1(\Omega,\mathbb H)$ is defined, respectively, as follows    
$${}^{{\psi}}\overline{\partial}f := \sum_{k=0}^2 \psi_k \partial_k f$$ 
and 
$${}^{{\psi}}\overline{\partial}_r f :=  \sum_{k=0}^2 \partial_k f \psi_k,$$ 
where  the partial differentiation operator   is abbreviated as $\partial_k f :=\displaystyle \frac{\partial f}{\partial x_k}$ for   $k=0,1,2$. In addition,   the three-dimensional Laplace operator $\Delta_{\mathbb R^3}$
  decomposes as
$${}^{{\psi}}\overline{\partial}\circ {}^{\overline{\psi}}\overline{\partial}={}^{\overline{\psi}}\overline{\partial} \circ {}^{{\psi}}\overline{\partial} = {}^{{\psi}}\overline{\partial}_r\circ {}^{\overline{\psi}}\overline{\partial}_r={}^{\overline{\psi}}\overline{\partial}_r \circ {}^{{\psi}}\overline{\partial}_r =\bigtriangleup_{\mathbb R^3}.$$
\end{definition}

\begin{definition}
A function $f \in C^1(\Omega,\mathbb H)$ is called left-$\psi$-hyperholomorphic  on $\Omega$ if ${}^{{\psi}}\overline{\partial}f=0$. In addition, if   ${}^{{\psi}}\overline{\partial}_r f=0$  on $\Omega$, then  $f$ is called   right-$\psi$-hyperholomorphic  on $\Omega$.
\end{definition} 

For the sake of completeness, we recall some definitions and results of  \cite{GPBN}.

\begin{notation}
The quaternionic differential form  $ {}^{\psi}\eta_{x}:= \sum_{k=0}^2 (-1)^k \psi_k d\hat{x}_k$ represents the $2$-dimensional area differential in $\bR^3$, where  $d\hat{x}_i  = dx_0 \wedge dx_1 \wedge dx_2 $ omitting  $dx_i$ for $i=0,1,2$. In addition, $|\eta_{x}| = dS_x$ is the differential element of $2$-dimensional surface.  The function: 
\[K_{\psi}(x) = \dfrac{\overline{x}}{4\pi |x|^{3}}, \quad \forall x\in \mathcal{A} \setminus \{0\},\]
where $\overline{x} = x_0 \psi_0- x_1\psi_1- x_2\psi_2$, is the ${\psi}$-hyperholomorphic Cauchy-Riemann kernel whose basic properties are $K_{\psi}\in C^{\infty} (\mathbb R^3\setminus \{ 0\})$ and  $ {}^{\psi}\overline{\partial} K_{\psi} = {}^{\psi}\overline{\partial}_{r} K_{\psi} =  0 $ on $\mathbb  R^3\setminus \{  0\}$. See \cite{GN1, GN2, S, SV1, SV2}. 
\end{notation}

\begin{proposition}\label{SBPF} 
If $f,g  \in C^1(\Omega,\mathbb H)$, then
\begin{align*}
 \int_{\partial \Omega} g{}^{\psi}\eta_x f &=  \int_{\Omega} \left[g({}^{\psi}\overline{\partial}  f)  + ({}^{\psi}\overline{\partial}_r  g) f\right]dx,
\end{align*}
\begin{align*}
&\int_{\partial\Omega }\left[ K_{\psi}(\tau -x){}^{\psi}\eta_{\tau} f(\tau)+g(\tau) {}^{\psi} \eta_{\tau} K_{\psi}(\tau -x)\right] \nonumber \\
&-\int_{\Omega }\left[ K_{\psi}(\tau-x) ({}^{\psi}\overline{\partial}f)(\tau)+({}^{\psi}\overline{\partial}_{r}g)(\tau) K_{\psi}(\tau-x)\right] d\tau \nonumber  \\
=&
\begin{cases}
f(x)+g(x), & x\in\Omega 
\\
0, & x\in\cA \setminus\overline{\Omega}.
\end{cases}
\end{align*} 
\end{proposition}
\begin{definition} The basic Möbius transformations on $\cA$ are the following: 
\begin{itemize}
\item Translation, $T_1(x): =   x+ u $, where $u\in \cA$, and  $T(\infty)=\infty$.
 
\item Dilation,  $T_2(x) := \lambda x$, where $\lambda >0$, and  $T(\infty)=\infty$.

\item Rotation, $T_3(x) := r  x \psi_3 \overline{r} \, \overline{\psi}_3 $, where $r\in \mathbb H$ with $\|r\|=1$,   and  $T(\infty)=\infty$.
One can verify that $T_3$ maps $\mathcal{A}$ to itself, preserves the norm, and its Jacobian is an orthogonal matrix with determinant~$1$; hence $T_3$ is indeed a rotation of $\mathcal{A}\cong\mathbb{R}^3$. The representation used here produces the same family of rotations as in~\cite{GN1}, but in the opposite orientation.

\item Inversion, $T_4 (x):= \dfrac{\overline{x}}{|x|^2}$, for all $x\in\cA\setminus\{ 0 \}$,   $T(0)=\infty$  and  $T(\infty)=0$.
\end{itemize}
\end{definition}

The M\"obius transformations on $\cA$, or {conformal} mapping  on $\cA$, in general are generated by the composition of a finite number of the previous four  transformations. Indeed, 
any  M\"obius transformation $T$ is represented by    
  \begin{align}\label{MoebiusR3} 
T(x)= (ax+b)(cx+d)^{-1} ,   
\end{align}   
where   
  \[
  \left\{   
 \begin{array}{ll}  
    bd^{-1}  \in \cA \quad \text{and} \quad \dfrac{d^{-1}}{|d^{-1}|}  =  \psi_3 \dfrac{\overline{a}}{|a|} \overline{\psi}_3
   ,  & \textrm{   if  $c=0$,}  \\
        \dfrac{(b - ac^{-1}d)^{-1}}{|(b - ac^{-1}d)^{-1}|}  =  \psi_3 \dfrac{\overline{c}}{|c|}\overline{\psi}_3  ,  \quad \text{and} \quad ac^{-1}, d(b - ac^{-1}d)^{-1} \in\cA , 
     & \textrm{   if  $c\neq0$,} 
  \end{array}
  \right.
  \]
  
\begin{remark} An affine transformation on $\cA$  is the composition of a finite number of translations, dilations and rotations.  In addition, any affine transformation  on $\cA$ can be expressed as follows
  \begin{align}\label{AffineR3} 
T(x)= (ax+b) d^{-1} ,   
\end{align}  
 where   
  $ 
    bd^{-1}  \in \cA \quad \text{and} \quad \dfrac{d^{-1}}{|d^{-1}|}  =  \psi_3 \dfrac{\overline{a}}{|a|} \overline{\psi}_3
  $.
\end{remark}

\begin{definition}
Two domains $ \Omega,\Xi \subset \cA$  are called conformally equivalent domains, if there exists a M\"obius transformation
  $T$ such that $T(\Xi)=  \Omega $. In particular, if  $T$ is an affine transformation, then $\Omega$ and $\Xi$ are called affine equivalent domains. 
\end{definition}

\begin{proposition}\label{conformableproperyofD}
Let $\Omega, \Xi \subset \cA $ be two conformally equivalent domains and  let $T$ be  a Möbius transformation   given by 
 \eqref{MoebiusR3} such that $T(\Omega) =\Xi$.
 Define 
 \begin{align*}
 A_T(x):= &\left\{  \begin{array}{ll}  \psi_3 \dfrac{\overline{a}}{|a|}\overline{ \psi_3} , & \ \ \textrm{if} \ \ c=0, \\  
  \psi_3 \dfrac{\overline{c}}{|c|}\overline{\psi_3}  \dfrac{\overline{cxv^{-1} + dv^{-1}}}{|cxv^{-1} + dv^{-1}|^3} , & \ \ \textrm{if} \ \ c\neq 0,  \end{array} \right. \\
B_T(x):= &\left\{  \begin{array}{ll}   |d^{-1}|\overline{a}  , & \ \ \textrm{if} \ \ c=0 ,\\
 - |v^{-1}| \overline{c} \dfrac{cxv^{-1} + dv^{-1}}{|cxv^{-1} + dv^{-1}|^5} , & \ \ \textrm{if} \ \ c\neq 0,  \end{array} \right.
  \end{align*}
 where $v:= b  -  ac^{-1}d$. Then 
  \begin{align}
  \label{invarianCauchy}
 {}^{{\psi}}\overline{\partial} [A_{T} f\circ T] = 
 B_{T}  \, {}^{{\psi}}\overline{\partial}[f]\circ T, \quad  \textrm{on } \ \ \Omega, 
  \end{align}
    for  all $f\in C^{1}(\Xi,\mathbb H)$. 
\end{proposition}

From~\eqref{invarianCauchy}, if $f$ is $\psi$-monogenic on $\Xi$, then $A_T\,f\circ T$ is also $\psi$-monogenic, since the right-hand
side vanishes. The factor $A_T$ thus plays the role of a \emph{conformal weight factor}, as already noted in~\cite{GN1}, where,
for $c\neq0$, it is given by $$ \frac{\overline{cx+d}}{|cx+d|^3}.$$
The difference with this expression lies in the constant $\dfrac{\overline{v^{-1}}}{|v^{-1}|^3}$ and the explicit rotation factor
$\psi_3\dfrac{\bar{c}}{|c|}\bar{\psi}_3$.

\subsection{Fractional proportional $\psi$-Cauchy-Riemann operators}

This subsection shows concepts and results demonstrated in   \cite{GPBN}. 

\begin{definition}\label{def000}
Given $a:=\sum_{k=0}^2\psi_k a_k$ and $b:=\sum_{k=0}^2\psi_k b_k \in\mathbb R^3$ such that $a_k< b_k$, for   $k=0,1,2$. Denote 
$${J_a^b }:=  \left\{  \sum_{k=0}^2\psi_k x_k \in \mathbb R^3 \ \mid \ a_k< x_k < b_k, \  \ k=0,1,2\right\}.$$ 
We temporarily fix the element $y= \sum_{k=0}^2\psi_k y_k \in J_a^b$.
\\
Given  $\sigma:=\sum_{k=0}^2\psi_k \sigma_k $ and  ${\alpha}:= (\alpha_0, \alpha_1,\alpha_2), \ {\beta}:=(\beta_0, \beta_1, \beta_2) \in \mathbb C^3$ such that 
$0< \Re \alpha_k, \Re\beta_k ,\sigma_k < 1$ for $k=0,1,2$.
\end{definition}
 
 \begin{definition}\label{integraoper}
Let $\varphi\in C^1(\overline{J_a^b},\mathbb R)$ be such that $\displaystyle\frac{\partial }{\partial x_i} \varphi(y_0, \dots, x_i, \dots y_2) > 0$ for all $x_i \in [a_i,b_i]$ and $i=0,1,2$. 
Denote  $$D\varphi(x,y):=\sum_{i=0}^2\displaystyle\frac{\partial }{\partial x_i} \varphi(y_0, \dots, x_i, \dots y_2), \quad \forall x\in \overline{J_a^b}.$$ 
The real linear space $AC^1(J_a^b,\mathbb H)$ consists of all functions ${f}=\sum_{i=0}^3\psi_i f_i$, where $f_i: J_a^b\to \mathbb R$ satisfies that the mapping   
$$x_j \mapsto f_i(y_0,\dots,x_j,\dots, y_2)$$  
belongs to $AC^1((a_j, b_j), \mathbb R)$ for all $y\in J_a^b$, $ j=0,1,2,$ and $ i=0,1,2,3$.  
\\ 
The left fractional proportional integral type operator of $ f \in  AC^1(J_a^b,\mathbb H)  $ with respect to $\varphi$ with order $\alpha$ and proportion   $\sigma$  is defined by
\begin{align*}
({}_a {  I}^{\alpha,\sigma,\varphi} f) (x,y): = &\sum_{i=0}^2 ({}_{a_i} I^{\alpha_i,\sigma_i,\varphi_i} f) (y_0, \dots, x_i ,\dots, y_2), 
\end{align*}
where $\varphi_i(x_i):= \varphi(y_0, \dots, x_i ,\dots, y_2)$ for all $x_i \in [a_i,b_i]$ for $i=0,1,2$. \\
The right fractional proportional integral type operator acts on $f$ as follows:  
  \begin{align*}
 ( {I}_b^{\alpha,\sigma,\varphi} f) (x,y) :=  \sum _{i=0}^2 (  I_{b_i}^{\alpha_i,\sigma_i,\varphi_i} f)(y_0, \dots, x_i ,\dots, y_2). 
 \end{align*}
\end{definition}

\begin{definition}
The quaternionic  left  and   right   fractional proportional $\psi$-Cauchy-Riemann type operators  with respect to  $\varphi$ with order $\alpha$ and proportion   $\sigma$ are given by    
\begin{align*}
({}^{\psi}_a{\overline{\partial}}^{\alpha, \sigma, \varphi }f)(x,y):= & (1-\sigma)  ({}_a{  I}^{1-\alpha, \sigma, \varphi}f)(x,y)  +
 \sigma \frac{{}^{\psi}\overline{\partial} ({}_a{I}^{1-\alpha,  \sigma, \varphi}  f )  (x,y)}{ D\varphi(x,y) },\\
 ( {}^{\psi} \overline{\partial}_b^{\alpha, \sigma, \varphi }   f)  (x,y)  
:=  &
  (1-\sigma)  ( {I}_b^{1-\alpha,  \sigma, \varphi}  f )  (x,y) +
     \sigma \frac{ {}^{\psi}\overline{\partial} ({I}_b^{1-\alpha,  \sigma, \varphi}  f )  (x,y)}{ D\varphi(x,y) } ,   
\end{align*} 
for all $ f \in  AC^1(J_a^b,\mathbb H)  $, respectively, where $1-\alpha := ( 1-\alpha_0, 1-\alpha_1, 1-\alpha_2 )\in \mathbb C^3 $. The partial derivatives of ${}^{\psi}\overline{\partial}$ are computed  with respect to real components of $x$.   
\\
The right versions of the previous operators are the following: 
\begin{align*}
 ( {}^{\psi}_a\overline{\partial}_r^{\alpha, \sigma, \varphi } f)  (x,y)  =  & 
   ({}_a{ I}^{1-\alpha,  \sigma, \varphi}  f )  (x,y)   (1-\sigma)   +
\frac{    {}^{\psi}\overline{\partial}_r ({}_a{I}^{1-\alpha,  \sigma, \varphi}  f )  (x,q)  }{ D\varphi(x,y) }  \sigma, \\
 (  {}^{\psi}\overline{\partial}_{r,b}^{\alpha, \sigma, \varphi }    f )  (x,y)  =   & 
    ( { I}_b^{1-\alpha,  \sigma, \varphi}  f )  (x,y)   (1-\sigma) +
\frac{     {}^{\psi}\overline{\partial}_r ({ I}_b^{1-\alpha,  \sigma, \varphi}  f )  (x,y) 
 }{ D\varphi(x,y) }\sigma,  
\end{align*}
respectively.
\end{definition}

\begin{definition} 
The quaternionic right module of 
left-$\alpha$-$\varphi$-fractional $\sigma$-proportional 
\\
$\psi$-hyperholomorphic functions on  ${J_a^b }$  is denoted by   
${}_a^{\psi}{\mathcal M}^{\alpha,\sigma, \varphi} (J_a^b)$  and consists of functions  $f \in AC^1(J_a^b,\mathbb H)$   such that $x\mapsto ({}_a{ I}^{\alpha,\sigma, \varphi}f)(x,y)$ belong  to $ C^1(\overline{J_a^b}, \mathbb H)$ for all $y\in J_a^b$ and 
$$({}^{\psi}_a{\overline{\partial}}^{\alpha,\sigma, \varphi}  f )(x,y)  =0,$$  
for all $x\in J_a^b$.  
\\
Similarly, the quaternionic left module ${}_a^{\psi}{\mathcal M}_r^{\alpha,\sigma, \varphi} (J_a^b)$ consists of functions $f \in AC^1(J_a^b,\mathbb H)$ such that $x\mapsto ({}_a{ I}^{\alpha,\sigma, \varphi}f)(x,y)$ belong  to $C^1(\overline{J_a^b}, \mathbb H)$ for all $y\in J_a^b$ and 
$$({}^{\psi}_a{\overline{\partial}}_r^{\alpha,\sigma, \varphi}f)(x,y) = 0,$$  
for all $x\in J_a^b$.  In addition, $f$  is called  right-$\alpha$-$\varphi$-fractional $\sigma$-proportional $\psi$-hyperholomorphic function on ${J_a^b}$ (r-$\alpha$-$\varphi$-$\sigma$-$\psi$-hyperholomorphic function on ${J_a^b}$, for short).
\end{definition}

\section{Main results}\label{MR}
Given the $\psi$-Cauchy-Riemann type operator ${}^\psi \overline{\partial}$ and $\sigma$ according to Definition \ref{def000}, 
as we will use some results from \cite{GPBN}, we will also assume        that there exist $\lambda_0, \lambda_1, \lambda_2 \in C_1 (\overline{J_a^b}, \bR)$   such that

\begin{align}\label{equaHyp}
   {}^{\psi} \overline{\partial} \left[ \sum_{k=0}^2 \lambda_k (\xi) \right]   =   \sum_{j=0}^2 \psi_j \dfrac{\partial}{\partial \xi_j} \sum_{k=0}^2 \lambda_k (\xi)   =   D \varphi (\xi, \zeta ) \sigma^{-1}(1-\sigma )
   ,
   \end{align}
  as assumed in   paper \cite{GPBN}.
\begin{remark}    
An important consequence of the identity \eqref{equaHyp} is 
    \begin{align} \label{DirectComputations}
   {}^{\psi}  \overline{\partial}_\xi \left[  e^{\sum_{k=0}^2 \lambda_k(\xi)} ({}_aI^{1-\alpha , \sigma , \varphi} f)(\xi,\zeta ) \right]   &=   D {\varphi} (\xi,\zeta) \sigma^{-1} {}_a^{\psi}\overline{\partial}_{\xi}^{\alpha , \sigma , \varphi} [f] (\xi, \zeta ) e^{ \sum_{k=0}^2 \lambda_k (\xi)},
    \\
   {}^{\psi}  \overline{\partial}_{r,\xi} \left[  e^{\sum_{k=0}^2 \lambda_k(\xi)} ({}_aI^{1-\alpha , \sigma , \varphi} f)(\xi,\zeta ) \right]   &=   {}_a^{\psi}\overline{\partial}_{r,\xi}^{\alpha , \sigma , \varphi} [f] (\xi, \zeta ) \sigma^{-1} D {\varphi} (\xi,\zeta)   e^{ \sum_{k=0}^2 \lambda_k (\xi) }, \nonumber
    \end{align}
    for all $f\in AC^1(J_a^b,\mathbb H)$, which are obtained  by direct computations (see \cite{GPBN}). 
\end{remark}

In the classical setting, the covariance identity~\eqref{invarianCauchy} expresses a clean relationship between ${}^{\psi}\overline{\partial}$ and
composition with $T$ due to the simple chain rule: $\partial_x[h(T(x))]=T'(x)\cdot h'(T(x))$. In the fractional proportional setting, the fractional integral ${}_aI^{1-\alpha,\sigma,\varphi}$ is not directly compatible with composition with $T$: the coordinate-wise structure of ${}_aI^{1-\alpha,\sigma,\varphi}$ interacts with the mixing of coordinates by $T$ in a way that does not reduce to a simple multiplicative factor. It is therefore necessary to introduce a new operator that incorporates $T$ directly into its definition.

\begin{definition}\label{DEF01}
Let $T$ be  a  Möbius transformation on $\cA$  given by  \eqref{MoebiusR3} and let $\Xi\subset \cA$ be such that  $T(\Xi)=J_a^b$. The quaternionic  left    fractional proportional $\psi$-Cauchy-Riemann type operators  with respect to  $\varphi$ with order $\alpha$, proportion   $\sigma$ and associated to $T$ is defined by 
\begin{align*} & {}_a^{\psi}\overline{\partial}_{x}^{\alpha , \sigma , \varphi, T} [g ] (x,y)  
\\
:=& 
    {}^{\psi} \overline{\partial}_x \left[\sum_{k=0}^2 \lambda_k \circ T(x)\right] {}_a I^{1-\alpha ,\sigma ,\varphi }[g]( T(x), T (y))   +   {}^{\psi}\overline{\partial}_x \bigg( {}_a I^{1-\alpha ,\sigma ,\varphi }[g ]( T(x), T (y))  \bigg),
\end{align*}
for all $g\in AC^{1}(J_a^b,\mathbb H)$. \end{definition}

\begin{remark}
By Definition \ref{DEF01} for $T=I$ (the identity mapping), ${}_a^{\psi}\overline{\partial} ^{\alpha , \sigma , \varphi, I}$ is directly related to 
${}_a^{\psi}\overline{\partial}^{\alpha , \sigma , \varphi}$  as follows
$${}_a^{\psi}\overline{\partial}^{\alpha , \sigma , \varphi, I}  = D\varphi (\xi, \zeta) \sigma ^{-1}  \  {}_a^{\psi}\overline{\partial} ^{\alpha , \sigma , \varphi}.$$
Moreover, the compositions ${}_a^{\psi}\overline{\partial}^{\alpha , \sigma , \varphi, T}  \circ  \overline{{}_a^{\psi}\overline{\partial} ^{\alpha , \sigma , \varphi, T}}$ and $\overline{{}_a^{\psi}\overline{\partial}^{\alpha , \sigma , \varphi, T}} \circ {}_a^{\psi}\overline{\partial} ^{\alpha , \sigma , \varphi, T}$  are extensions of the Laplace operator in a sense of the  quaternionic fractional analysis.
\end{remark}
 
The next proposition relates ${}_a^{\psi}\overline{\partial} ^{\alpha , \sigma , \varphi, T}$ with  ${}_a^{\psi}\overline{\partial}$. 

\begin{proposition}\label{prop1t}
Let $T$ be  a  Möbius transformation on $\cA$  given by  \eqref{MoebiusR3} and let $\Xi\subset \cA$ be such that  $T(\Xi)=J_a^b$. 
Then 
\begin{align*}
   {}^{\psi}\overline{\partial}_x \left\{   e^{\sum_{k=0}^2 \lambda_k \circ T(x)} {}_a I^{1-\alpha ,\sigma ,\varphi }[g]( T(x), T( y)   )   \right\} =  e^{ \sum_{k=0}^2 \lambda_{k}\circ T(x)  } \   {}_a^{\psi}\overline{\partial}_{x}^{\alpha , \sigma , \varphi, T}[g]( x,y   ) ,
\end{align*}
 for all $g\in C^1(J_a^b, \mathbb H) $.
\end{proposition}
\begin{proof}
 \begin{align*}
& {}^{\psi}\overline{\partial}_x \left\{   e^{\sum_{k=0}^2 \lambda_k \circ T(x)} {}_a I^{1-\alpha ,\sigma ,\varphi }[g]( T(x), T( y)   )   \right\} =\\
 & {}^{\psi}\overline{\partial}_x \left\{   e^{\sum_{k=0}^2 \lambda_k \circ T(x)}   \right\}
{}_a I^{1-\alpha ,\sigma ,\varphi }[g]( T(x), T( y)   )  
  +  e^{\sum_{k=0}^2 \lambda_k \circ T(x)} {}^{\psi}\overline{\partial}_x \left\{   {}_a I^{1-\alpha ,\sigma ,\varphi }[g]( T(x), T( y)   )   \right\} = \\
  & \bigg\{ {}^{\psi}\overline{\partial}_x \left[   \sum_{k=0}^2 \lambda_k \circ T(x)    \right]
{}_a I^{1-\alpha ,\sigma ,\varphi }[g]( T(x), T( y)   )    
   +   {}^{\psi}\overline{\partial}_x \bigg[  {}_a I^{1-\alpha ,\sigma ,\varphi }[g]( T(x), T( y)   )   \bigg] \bigg\} e^{\sum_{k=0}^2 \lambda_k \circ T(x)}.
\end{align*}
\end{proof}

\begin{remark}\label{RelDI}
Define  
$${\bf D}_a^{1-\alpha, \sigma, \varphi} = \sum_{i=0}^2 {}_{a_i} D^{1-\alpha_i,\sigma_i,\varphi_i} ,$$
where ${}_{a_i} D^{1-\alpha_i,\sigma_i,\varphi_i}$ is the fractional proportional partial derivative  with order $1-\alpha_i$, with respect to $\varphi_i$ in the real component $x_i$ of $x$, for $i=0,1,2$, according to \eqref{DerFracPropRespecFuntL}. Then  
\begin{align*}
& {\bf D}_a^{1-\alpha, \sigma, \varphi} ({}_a {  I}^{\alpha,\sigma,\varphi} f) (x,y)  =  \sum_{j=0}^2 {}_{a_j} D^{1-\alpha_j,\sigma_j,\varphi_j}  \sum_{i=0}^2 ({}_{a_i} I^{\alpha_i,\sigma_i,\varphi_i} f) (y_0, \dots, x_i ,\dots, y_2) \\
 = & \sum_{i=0}^2  f  (y_0, \dots, x_i ,\dots, y_2) + 
 { \displaystyle \sum_{
{   \begin{subarray}{c} i,j=0 \\
    i\neq j \end{subarray} }}^2 }
 {}_{a_j} D^{1-\alpha_j,\sigma_j,\varphi_j}   ({}_{a_i} I^{\alpha_i,\sigma_i,\varphi_i} f) (y_0, \dots, x_i ,\dots, y_2), 
 \end{align*}
for all  $f  \in AC^1(J_a^b,\mathbb H)$,  where 
 $ {}_a {  I}^{\alpha,\sigma,\varphi}$ is given by Definition \ref{integraoper} and \eqref{FundTheoFPW} has been used.
\end{remark}

\begin{proposition}[Stokes and Borel-Pompeiu type formulas associated to  $ {}_a^{\psi}\overline{\partial}_{x}^{\alpha , \sigma , \varphi, T}$] \label{prop03} 
Let $T$ be  a  Möbius transformation on $\cA$  given by  \eqref{MoebiusR3} and let $\Xi\subset \cA$ be such that  $T(\Xi)=J_a^b$.  Given $g  \in AC^1(J_a^b,\mathbb H)$, we have
\begin{align*}
\int_{\partial \Xi}  {}^{\psi}\eta_x 
    e^{\sum_{k=0}^2 \lambda_k \circ T(x)} 
    & {}_a I^{1-\alpha ,\sigma ,\varphi }[g]( T(x), T( y)   ) 
    \\
   & =
\int_{\Xi}  {}_a^{\psi}\overline{\partial}_{x}^{\alpha , \sigma , \varphi, T}[g]( x,y   )       e^{\sum_{k=0}^2 
 \lambda_k \circ T(x)} dx   
\end{align*} 
and 
 \begin{align*}
&{\bf D}_a^{1-\alpha, \sigma, \varphi} \bigg[\int_{\partial\Xi }   e^{\sum_{k=0}^2[
 \lambda_k \circ T(\tau)  -\lambda_k \circ T(x)]}  K_{\psi}(\tau -x) {}^{\psi}\eta_{\tau}     {}_a I^{1-\alpha ,\sigma ,\varphi }[g]( T(\tau), T( y)   )  \bigg]\\
& - 
{\bf D}_a^{1-\alpha, \sigma, \varphi} \bigg[ \int_{\Xi }
  e^{\sum_{k=0}^2[
 \lambda_k \circ T(\tau)  -\lambda_k \circ T(x)]} 
 K_{\psi}(\tau-x)   {}_a^{\psi}\overline{\partial}_{\tau}^{\alpha , \sigma , \varphi, T}[g](  \tau ,  y   )     d\tau \bigg] \\ 
  &=
\begin{cases}
 \displaystyle  \sum_{i=0}^2 g (T(y)_0, \dots, T(x)_i ,\dots, T(y)_2)  \\
    +  
  { \displaystyle \sum_{
{     \begin{subarray}{c} i,j=0 \\
    i\neq j \end{subarray} }}^2 } {}_{a_j} D^{1-\alpha_j,\sigma_j,\varphi_j}  ({}_{a_i} I^{1-\alpha_i,\sigma_i,\varphi_i} g) (T(y)_0, \dots, T(x)_i ,\dots, T(y)_2), & x\in\Xi, 
\\
 0, & x\in\cA \setminus\overline{\Xi},
\end{cases}
\end{align*}
 \end{proposition}
\begin{proof}
From Proposition 
\ref{SBPF} we see that 
 \begin{align*}
& \int_{\partial \Xi}  {}^{\psi}\eta_x h(x) =  \int_{\Xi}  {}^{\psi}\overline{\partial}  h(x) dx, \\ 
&\int_{\partial\Xi }  K_{\psi}(\tau -x){}^{\psi}\eta_{\tau} h(\tau) - \int_{\Xi }
 K_{\psi}(\tau-x) ({}^{\psi}\overline{\partial}h)(\tau)  d\tau \nonumber  =
\begin{cases}
h(x) , & x\in\Xi, 
\\
0, & x\in\cA \setminus\overline{\Xi}.
\end{cases}
\end{align*} 
for all $h\in  AC^1(\Xi,\mathbb H)$. In particular, 
 if   $$h(x) =    e^{\sum_{k=0}^2 \lambda_k \circ T(x)} {}_a I^{1-\alpha ,\sigma ,\varphi }[g]( T(x), T( y)   ) ,$$
then 
 \begin{align*}
  \int_{\partial \Xi}  {}^{\psi}\eta_x 
    e^{\sum_{k=0}^2 \lambda_k \circ T(x)} {}_a I^{1-\alpha ,\sigma ,\varphi }[g]( T(x), T( y)   )    
    = &  \int_{\Xi}  {}^{\psi}\overline{\partial} \bigg[    e^{\sum_{k=0}^2 \lambda_k \circ T(x)} {}_a I^{1-\alpha ,\sigma ,\varphi }[g]( T(x), T( y)   ) \bigg]  dx  \\ 
\end{align*} 
 and
  \begin{align*}
&\int_{\partial\Xi }  K_{\psi}(\tau -x){}^{\psi}\eta_{\tau}    e^{\sum_{k=0}^2 \lambda_k \circ T(\tau)} {}_a I^{1-\alpha ,\sigma ,\varphi }[g]( T(\tau), T( y)   )  \\
& - \int_{\Xi }
 K_{\psi}(\tau-x)  {}^{\psi}\overline{\partial} \bigg[  e^{\sum_{k=0}^2 \lambda_k \circ T(\tau)} {}_a I^{1-\alpha ,\sigma ,\varphi }[g]( T(\tau), T( y)   ) \bigg] d\tau \\ 
  &=
\begin{cases}
    e^{\sum_{k=0}^2 \lambda_k \circ T(x)} {}_a I^{1-\alpha ,\sigma ,\varphi }[g]( T(x), T( y)   )  , & x\in\Xi, 
\\
0, & x\in\cA \setminus\overline{\Xi}.
\end{cases}
\end{align*} 
Multiplying by $e^{-\sum_{k=0}^2 \lambda_k \circ T(x)}$ on both sides of the previous formulas and using Proposition \ref{prop1t}  we have that 
 \begin{align*}
\int_{\partial \Xi}  {}^{\psi}\eta_x 
    e^{\sum_{k=0}^2 \lambda_k \circ T(x)} & {}_a I^{1-\alpha ,\sigma ,\varphi }[g]( T(x), T( y)   )   
    \\
    & = \int_{\Xi}  {}_a^{\psi}\overline{\partial}_{x}^{\alpha , \sigma , \varphi, T}[g]( x,y  )       e^{\sum_{k=0}^2 
 \lambda_k \circ T(x) } dx   
\end{align*}
 and
\begin{align*}
&\int_{\partial\Xi }   e^{\sum_{k=0}^2[
 \lambda_k \circ T(\tau)  -\lambda_k \circ T(x)]}  K_{\psi}(\tau -x) {}^{\psi}\eta_{\tau}     {}_a I^{1-\alpha ,\sigma ,\varphi }[g]( T(\tau), T( y)   )  \\
& - \int_{\Xi }
  e^{\sum_{k=0}^2[
 \lambda_k \circ T(\tau)  -\lambda_k \circ T(x)]} 
 K_{\psi}(\tau-x)   {}_a^{\psi}\overline{\partial}_{\tau}^{\alpha , \sigma , \varphi, T}[g](  \tau,  y    )     d\tau \\ 
  &=
\begin{cases}
   {}_a I^{1-\alpha ,\sigma ,\varphi }[g]( T(x), T( y)   )  , & x\in\Xi, 
\\
0, & x\in\cA \setminus\overline{\Xi}.
\end{cases}
\end{align*} 
Apply the operator $ {\bf D}_a^{1-\alpha, \sigma, \varphi}$ to both sides of the last formula  to obtain that   
 \begin{align*}
&{\bf D}_a^{1-\alpha, \sigma, \varphi} \bigg[\int_{\partial\Xi }   e^{\sum_{k=0}^2[
 \lambda_k \circ T(\tau)  -\lambda_k \circ T(x)]}  K_{\psi}(\tau -x) {}^{\psi}\eta_{\tau}     {}_a I^{1-\alpha ,\sigma ,\varphi }[g]( T(\tau), T( y)   )  \bigg]\\
& - 
{\bf D}_a^{1-\alpha, \sigma, \varphi} \bigg[ \int_{\Xi }
  e^{\sum_{k=0}^2[
 \lambda_k \circ T(\tau)  -\lambda_k \circ T(x)]} 
 K_{\psi}(\tau-x)   {}_a^{\psi}\overline{\partial}_{\tau}^{\alpha , \sigma , \varphi, T}[g](  \tau,  y    )     d\tau \bigg] \\ 
  &=
\begin{cases}
{\bf D}_a^{1-\alpha, \sigma, \varphi}   {}_a I^{1-\alpha ,\sigma ,\varphi }[g]( T(x), T( y)   )  , & x\in\Xi, 
\\
 0, & x\in\cA \setminus\overline{\Xi}.
\end{cases}
\end{align*} 
 Therefore, the computations presented  in Remark \ref{RelDI} allow us to simplify the right-hand side as follows:
 \begin{align*}
&{\bf D}_a^{1-\alpha, \sigma, \varphi} \bigg[\int_{\partial\Xi }   e^{\sum_{k=0}^2[
 \lambda_k \circ T(\tau)  -\lambda_k \circ T(x)]}  K_{\psi}(\tau -x) {}^{\psi}\eta_{\tau}     {}_a I^{1-\alpha ,\sigma ,\varphi }[g]( T(\tau), T( y)   )  \bigg]\\
& - 
{\bf D}_a^{1-\alpha, \sigma, \varphi} \bigg[ \int_{\Xi }
  e^{\sum_{k=0}^2[
 \lambda_k \circ T(\tau)  -\lambda_k \circ T(x)]} 
 K_{\psi}(\tau-x)   {}_a^{\psi}\overline{\partial}_{\tau}^{\alpha , \sigma , \varphi, T}[g](  \tau , y   )     d\tau \bigg] \\ 
  &=
\begin{cases}
  \displaystyle \sum_{i=0}^2 g (T(y)_0, \dots, T(x)_i ,\dots, T(y)_2)  \\
    +  
  { \displaystyle \sum_{
{   \begin{subarray}{c} i,j=0 \\
    i\neq j \end{subarray} }}^2 } {}_{a_j} D^{1-\alpha_j,\sigma_j,\varphi_j}  ({}_{a_i} I^{1-\alpha_i,\sigma_i,\varphi_i} g) (T(y)_0, \dots, T(x)_i ,\dots, T(y)_2), & x\in\Xi, 
\\
 0,\hspace{3cm} x\in\cA \setminus\overline{\Xi}. & 
\end{cases}
\end{align*}
\end{proof}

 \begin{definition}\label{Def5}
 Let $T$ be  a  Möbius transformation on $\cA$  given by  \eqref{MoebiusR3} and let $\Xi\subset \cA$ be such that  $T(\Xi)=J_a^b$. The quaternionic right module of 
 left-$\alpha$-$\varphi$-fractional $\sigma$-proportional $\psi$-hyperholomorphic functions associated to $T$ on  ${\Xi}$  is denoted by 
 $${}^{\psi}_a{{\mathcal M}}^{\alpha,\sigma, \varphi, T} (\Xi):= \{h\in AC^1(\Xi,\mathbb H) \ \mid \ h\circ T^{-1}\in \textrm{Ker} ({}_a^{\psi}\overline{\partial}_{x}^{\alpha, \sigma, \varphi, T})\cap C^1(J_a^b, \mathbb H)\}.$$
 \end{definition}
 
\begin{remark}
Note that $ h\in  {}^{\psi}_a{{\mathcal M}}^{\alpha,\sigma, \varphi, T} (\Xi)$ if and only if
$${}^{\psi}\overline{\partial}_x \bigg( {}_a I^{1-\alpha ,\sigma ,\varphi }[h\circ T^{-1} ]( T(x), T (y))\bigg) = - {}^{\psi} \overline{\partial}_x \left[\sum_{k=0}^2 \lambda_k \circ T(x)\right] {}_a I^{1-\alpha ,\sigma ,\varphi }[h\circ T^{-1}]( T(x), T (y)),$$ 
which provides an extension of the Cauchy-Riemann equations to this setting.
 \end{remark}

 \begin{corollary}\label{cor110}[Cauchy type theorem and  formula in  ${}^{\psi}_a{{\mathcal M}}^{\alpha,\sigma, \varphi, T} (\Xi)$]
 Let $T$ be  a  Möbius transformation on $\cA$  given by  \eqref{MoebiusR3} and let $\Xi\subset \cA$ be such that  $T(\Xi)=J_a^b$.  If $h\in {}^{\psi}_a{{\mathcal M}}^{\alpha,\sigma, \varphi, T} (\Xi)$ then 
\begin{align*}
& \int_{\partial \Xi}  {}^{\psi}\eta_x 
    e^{\sum_{k=0}^2 \lambda_k \circ T(x)} {}_a I^{1-\alpha ,\sigma ,\varphi }[h\circ  T^{-1}]( T(x), T( y)   )   =0  
\end{align*} 
 and 
 \begin{align*}
&{\bf D}_a^{1-\alpha, \sigma, \varphi} \bigg[\int_{\partial\Xi }   e^{\sum_{k=0}^2[
 \lambda_k \circ T(\tau)  -\lambda_k \circ T(x)]}  K_{\psi}(\tau -x) {}^{\psi}\eta_{\tau}     {}_a I^{1-\alpha ,\sigma ,\varphi }[h\circ  T^{-1}]( T(\tau), T( y)   )  \bigg]\\
  &=
\begin{cases}
  \displaystyle \sum_{i=0}^2 h\circ  T^{-1} (T(y)_0, \dots, T(x)_i ,\dots, T(y)_2)  \\
    +  
  { \displaystyle \sum_{
{  \begin{subarray}{c} i,j=0 \\
    i\neq j \end{subarray} }}^2 } {}_{a_j} D^{1-\alpha_j,\sigma_j,\varphi_j}  ({}_{a_i} I^{1-\alpha_i,\sigma_i,\varphi_i} h\circ T^{-1}) (T(y)_0, \dots, T(x)_i ,\dots, T(y)_2), & x\in\Xi, 
\\
 0, \hspace{3cm} x\in\cA \setminus\overline{\Xi}, &
\end{cases}
\end{align*}
  \end{corollary}
\begin{proof}
Set $g= h\circ T^{-1}$ in Proposition \ref{prop03}.
\end{proof}

\begin{remark} If $T$ is a translation or a scalar dilation (i.e., $a=a_0\mathbf{e}_0$ in~\eqref{AffineR3}, so that the Jacobian of $T$ is a scalar multiple of the identity), then a direct computation shows that 
\[
 h\circ T^{-1} (T(y)_0, \dots, T(x)_i ,\dots, T(y)_2) =    h(y _0, \dots,  x _i ,\dots,  y _2), \hspace{0.5cm} i=0,1,2.
\]
For a general affine transformation whose Jacobian involves a rotation, this identity does not hold.\footnote{As a concrete counterexample, take
$T(x_0,x_1,x_2)=(-x_1,x_0,x_2)$ (a rotation in the $\psi_0$-$\psi_1$ plane), so that $T^{-1}(u_0,u_1,u_2)=(u_1,-u_0,u_2)$.
For $i=1$ one computes $(T(y)_0,T(x)_1,T(y)_2)=(-y_1,x_0,y_2)$, and hence $T^{-1}(-y_1,x_0,y_2)=(x_0,y_1,y_2)$,
whereas the right-hand side of the claimed identity would require $(y_0,x_1,y_2)$.}
Under this additional hypothesis, the Cauchy type formula in Corollary \ref{cor110} simplifies to
 \begin{align*}
&{\bf D}_a^{1-\alpha, \sigma, \varphi} \bigg[\int_{\partial\Xi }   e^{\sum_{k=0}^2[
 \lambda_k \circ T(\tau)  -\lambda_k \circ T(x)]}  K_{\psi}(\tau -x) {}^{\psi}\eta_{\tau}     {}_a I^{1-\alpha ,\sigma ,\varphi }[h\circ T^{-1}]( T(\tau), T( y)   )  \bigg]\\
   &=
\begin{cases}
  \displaystyle \sum_{i=0}^2 h ( y _0, \dots,  x _i ,\dots,  y _2)  \\
    +  
  { \displaystyle \sum_{
{  \begin{subarray}{c} i,j=0 \\
    i\neq j \end{subarray} }}^2 } {}_{a_j} D^{1-\alpha_j,\sigma_j,\varphi_j}  ({}_{a_i} I^{1-\alpha_i,\sigma_i,\varphi_i} h\circ T^{-1}) (T(y)_0, \dots, T(x)_i ,\dots, T(y)_2), \\
   {}\hspace{2cm} x\in\Xi , & 
\\
 0, {}\hspace{1.5cm} x\in\cA \setminus\overline{\Xi}.& 
\end{cases}
\end{align*}
\end{remark}

\begin{corollary}\label{cor14}
 Let $T$ be  an affine  transformation on $\cA$  
 given by 
 \eqref{AffineR3}  and let $\Xi\subset \cA$ be such that  $T(\Xi)=J_a^b$.  
 Then 
\begin{align*}  
{}^{\psi}\overline{\partial}_x & \left[  e^{\sum_{k=0}^2 \lambda_k \circ T(x)}  {}_a I^{1-\alpha ,\sigma ,\varphi }[A_T  f]( T(x), T( y)   )   \right] 
\\
&=   {}_a^{\psi}\overline{\partial}_{x}^{\alpha , \sigma , \varphi, T}[A_T f ]( x, y  )  e^{ \sum_{k=0}^2 \lambda_{k}\circ T(x)  },
\end{align*}
 for all $f\in C^1(J_a^b, \mathbb H) $.
\end{corollary}
\begin{proof}
Since $c=0$ for an affine transformation, $A_T$ is a quaternionic constant. The result then follows by applying Proposition~\ref{prop1t} with $g=A_Tf$. 
\end{proof}

\begin{theorem}\label{Afinneprop}[An affine  covariance type property of $ {}_a^{\psi}\overline{\partial}^{\alpha , \sigma , \varphi, T} $]  
Let $T$ be  an affine  transformation on $\cA$  
 given by 
 \eqref{AffineR3}  and let $\Xi\subset \cA$ be such that  $T(\Xi)=J_a^b$.  
  If  $f\in AC^1 (J_a^b, \mathbb H)$,  then   
 \begin{align*}
   {}_a^{\psi}\overline{\partial}_{x}^{\alpha , \sigma , \varphi, T} [A_T  f] ( x ,  y  )  e^{ \sum_{k=0}^2 \lambda_{k}\circ T(x)  } 
   = 
 \beta_T (x) {}_a^{\psi}\overline{\partial}_{\xi}^{ \alpha , \sigma , \varphi} [f] (\xi, T(y) ) e^{ \sum_{k=0}^2 \lambda_k (\xi)},    \end{align*}
where     $\xi=T(x)$ and $\beta_T(x) := B_T (x)  D{\varphi} (\xi, T(y)) \sigma^{-1}$. 
 \end{theorem}
 \begin{proof}
Using  Corollary \ref{cor14}, Proposition \ref{conformableproperyofD} and Equation \eqref{DirectComputations}, we have that
\begin{align*}
&   {}_a^{\psi}\overline{\partial}_{x}^{\alpha , \sigma , \varphi, T}[A_T f ]( x, y)  e^{ \sum_{k=0}^2 \lambda_{k}\circ T(x)  } = 
{}^{\psi}\overline{\partial}_x   \left[  e^{\sum_{k=0}^2 \lambda_k \circ T(x)} {}_a I^{1-\alpha ,\sigma ,\varphi} [A_T  f](T(x), T(y)) \right]   
\\
 =&  
 {}^{\psi}\overline{\partial}_x   \left[   A_T  (x) e^{\sum_{k=0}^2 \lambda_k \circ T(x)} {}_a I^{1-\alpha ,\sigma ,\varphi} f(T(x), T(y)) \right]   \\
 = &
{}^{\psi}\overline{\partial}_x   \left[   A_T (x) \left( e^{\sum_{k=0}^2 \lambda_k(T(x)) } {}_a I^{1-\alpha ,\sigma ,\varphi} f(T(x) ,  T(y) ) \right)  \right]   
\\
 = &
  B_T (x){}^{\psi}\overline{\partial}_\xi \left( e^{\sum_{k=0}^2 \lambda_k(\xi) } {}_a I^{1-\alpha ,\sigma ,\varphi} f(\xi , T(y)  ) \right)      \\
 = &
  B_T (x)  D{\varphi} (\xi, T(y))     \sigma^{-1} {}_a^{\psi}\overline{\partial}_{\xi}^{ \alpha , \sigma , \varphi} [f] (\xi, T(y) ) e^{ \sum_{k=0}^2 \lambda_k (\xi)}.
\end{align*}
\end{proof}

\begin{corollary}\label{cor1001}
Let $T$ be  an affine  transformation on $\cA$  
 given by 
 \eqref{AffineR3}  and let $\Xi\subset \cA$ be such that  $T(\Xi)=J_a^b$.  
Then 
$A_T h\in {}^{\psi}_a{{\mathcal M}}^{\alpha,\sigma, \varphi, T} (\Xi)$
iff $h \circ T^{-1} \in {}^{\psi}_a{{\mathcal M}}^{\alpha,\sigma, \varphi} (J_a^b)$.
 \end{corollary}
 \begin{proof} 
 Setting $f= h\circ T^{-1}$ in Theorem \ref{Afinneprop}, we obtain  
  \begin{align*}
   {}_a^{\psi}\overline{\partial}_{x}^{\alpha , \sigma , \varphi, T} [A_T  h\circ T^{-1}] ( x ,  y  )
   e^{ \sum_{k=0}^2 \lambda_{k}\circ T(x)  }  = 
 \beta_T (x) {}_a^{\psi}\overline{\partial}_{\xi}^{ \alpha , \sigma , \varphi} [ h\circ T^{-1}] (\xi, T(y) ) e^{ \sum_{k=0}^2 \lambda_k (\xi)}.    \end{align*}
Since $\beta_T(x)=B_T(x)\,D\varphi(\xi,T(y))\,\sigma^{-1}\neq 0$ and $e^{\sum_k\lambda_k(\xi)}\neq 0$, the right-hand side vanishes if and only if ${}_a^{\psi}\overline{\partial}_{\xi}^{\,\alpha,\sigma,\varphi}[h\circ T^{-1}](\xi,T(y))=0$, which is precisely the condition
$h\circ T^{-1}\in{}^{\psi}_a\mathcal{M}^{\alpha,\sigma,\varphi}(J_a^b)$.
 \end{proof}

\begin{remark}
The previous corollary is equivalent to the statement: 
$g\in {}^{\psi}_a{{\mathcal M}}^{\alpha,\sigma, \varphi, T} (\Xi)$ if and only if $(A_T)^{-1}  g \circ T^{-1}  \in {}^{\psi}_a{{\mathcal M}}^{\alpha,\sigma, \varphi} (J_a^b)$, which follows by setting $g = (A_T)^{-1} h$ and recalling that $A_T$ is a non-zero quaternionic constant. Therefore, the quaternionic right-linear  operator  
$$L_{_T}:  {}^{\psi}_a{{\mathcal M}}^{\alpha,\sigma, \varphi, T} (\Xi) \to {}^{\psi}_a{{\mathcal M}}^{\alpha,\sigma, \varphi} (J_a^b) , \hspace{0.5cm} L_{_T}[g]= (A_T)^{-1} g \circ T^{-1},   $$ 
is bijective, with inverse $L_{_T}^{-1}[f]= A_T  \  (f \circ T ) $  for all $ f\in  {}^{\psi}_a{{\mathcal M}}^{\alpha,\sigma, \varphi} (J_a^b)$.
 \\
Furthermore, if $T_1$ and $T_2$ are two affine transformations and $\Xi_1, \Xi_2\subset \cA$ are such that $T_i(\Xi_i)=J_a^b$ for $i=1,2$, then each
 $L_{_{T_i}}:  {}^{\psi}_a{{\mathcal M}}^{\alpha,\sigma, \varphi, T_i} (\Xi_i) \to {}^{\psi}_a{{\mathcal M}}^{\alpha,\sigma, \varphi} (J_a^b) $ 
is a quaternionic right-linear bijection, and the composition  
$$L_{_{T_1,T_2}} := L_{_{T_2}}^{-1} \circ L_{_{T_1}}:  {}^{\psi}_a{{\mathcal M}}^{\alpha,\sigma, \varphi, T_1} (\Xi_1) \to  {}^{\psi}_a{{\mathcal M}}^{\alpha,\sigma, \varphi, T_2} (\Xi_2)$$  
is likewise a quaternionic right-linear bijection. 
\end{remark}
 
\begin{remark}
We now use the theory of categories and functors, well known in abstract algebra, to give an interpretation of the covariance and affine invariant properties established in Theorem \ref{Afinneprop} and Corollary \ref{cor1001}.

\begin{enumerate}
\item 
By $\mathcal C_a^b$  we denote the category induced by $J_a^b$. The family of objects in  $\mathcal C_a^b$ is 
\[  \{ \Xi\subset \cA \ \mid \ \textrm{ there exists an affine transformation $T$ such that } T(\Xi)=J_a^b\}.\] 
Given two objects $\Xi_1$ and $\Xi_2$, the morphism is given by $ T_2^{-1}\circ T_1: \Xi_1 \to \Xi_2$, where $T_i(\Xi_i)=J_a^b$ for $i=1,2$. 

\item The category $\mathcal M_a^b$ is formed by the family of objects 
$$\{ {}^{\psi}_a{{\mathcal M}}^{\alpha,\sigma, \varphi, T} (\Xi)  \mid\  T  (\Xi)  =J_a^b \}$$
and given two objects   ${}^{\psi}_a{{\mathcal M}}^{\alpha,\sigma, \varphi, T_1} (\Xi_1)$ and 
 ${}^{\psi}_a{{\mathcal M}}^{\alpha,\sigma, \varphi, T_2} (\Xi_2)$, its morphism  is $L_{_{T_1,T_2}}$, according  to the previous remark.
\end{enumerate}
Therefore, the assignments 
\begin{align*} \Xi  & \  \mapsto \  {}^{\psi}_a{{\mathcal M}}^{\alpha,\sigma, \varphi, T} (\Xi) ,\hspace{0.5cm}  T(\Xi) = J_a^b  , \\
T_2^{-1}\circ T_1  & \ \mapsto L_{_{T_1,T_2}} := L_{_{T_2}}^{-1} \circ L_{_{T_1}},
\end{align*}
define a covariant functor from $\mathcal C_a^b$ to $\mathcal M_a^b$.
\end{remark}

\begin{remark}
 Operator    $ {}_a^{\psi}\overline{\partial}_{r}^{\alpha , \sigma , \varphi} $
 studied in \cite{GPBN}, together with the right version of generalized Cauchy-Riemann type operator $ {} ^{\psi}\overline{\partial}_{r} $, can be used to develop the function theory associated to the right operator    
\begin{align*}
 {}_a^{\psi}\overline{\partial}_{x,r}^{\alpha , \sigma , \varphi, T} [g ] (x,y) 
& := 
    {}_a I^{1-\alpha ,\sigma ,\varphi }[g]( T(x), T (y))  {}^{\psi} \overline{\partial}_{x,r} \left[\sum_{k=0}^2 \lambda_k \circ T(x)\right]  
    \\
    & \qquad +   {}^{\psi}\overline{\partial}_{x,r} \bigg( {}_a I^{1-\alpha ,\sigma ,\varphi }[g ]( T(x), T (y))  \bigg),
\end{align*}
for all $g\in AC^{1}(J_a^b,\mathbb H)$. An affine covariance type property of $ {}_a^{\psi}\overline{\partial}_r^{\alpha , \sigma , \varphi}$, analogous to Theorem \ref{Afinneprop}, can then be established.
\end{remark}

\begin{remark}[On the restriction to affine transformations]
Throughout this section, the covariance property of Theorem~\ref{Afinneprop} has been established only for affine
transformations ($c=0$). This restriction reflects a genuine obstruction specific to the fractional setting. 
In the classical theory, identity~\eqref{invarianCauchy} holds for the full M\"{o}bius group.
When $c\neq0$, $A_T=A_T(x)$ is non-constant. Even so, it enters the computation only through the classical product rule, which handles the
differentiation of $A_T(x)\cdot f(T(x))$ exactly, via the chain rule $\partial_x[h(T(x))]=T'(x)\cdot h'(T(x))$. 
In the fractional setting, extending the argument requires differentiating
$$A_T(x)\cdot{}_aI^{1-\alpha,\sigma,\varphi}[f](T(x),T(y)),$$ 
where $A_T$ sits outside a fractional integral whose limits depend on $x$.
Since $A_T$ is non-constant when $c\neq0$, it cannot be factored through
the integral. The resulting term
$${}^{\psi}\overline{\partial}_x[A_T(x)]\cdot{}_aI^{1-\alpha,\sigma,\varphi}[f](T(x),T(y))$$
does not simplify: unlike the classical case, fractional derivatives do
not satisfy a simple product rule, and Leibniz-type formulas for
fractional integrals take the form of infinite series (\cite{SKM}). The
affine case $c=0$ is precisely the one where $A_T$ is a quaternionic constant (see Corollary~\ref{cor14}), so this obstruction disappears entirely.

\end{remark}

\begin{remark}[On the conformal weight factor for general M\"{o}bius transformations]
The results of this paper naturally raise the question of what form the conformal weight factor $A_T$ and the covariance identity of
Theorem~\ref{Afinneprop} take when $c\neq0$. One expects the identity to become
$$
  {}_a^{\psi}\overline{\partial}_{x}^{\,\alpha,\sigma,\varphi,T}[A_Tf](x,y)\,
  e^{\sum_{k=0}^{2}\lambda_k\circ T(x)}
  = \beta_T(x)\,
    {}_a^{\psi}\overline{\partial}_{\xi}^{\,\alpha,\sigma,\varphi}[f](\xi,T(y))\,
    e^{\sum_{k=0}^{2}\lambda_k(\xi)}
    + R_T(x,y),
$$
where $R_T$ is a remainder generated by the fractional differentiation
of $A_T$ itself, vanishing when $c=0$. 
For the general case, the fractional Leibniz formula (\cite{SKM})
$$
  {}_{a_i}D^{\alpha_i}[A_T f](t)
  = \sum_{k=0}^{\infty}\binom{\alpha_i}{k}\,
    {}_{a_i}D^{\alpha_i-k}[A_T](t)\,D^k[f](t)
$$
provides a systematic method to compute $R_T$: the $k=0$ term is exactly the "naive" covariance term appearing in Theorem~\ref{Afinneprop}, while the terms $k\geq1$ constitute the remainder. Whether $A_T\cdot(f\circ T)$ remains in the module ${}^{\psi}_a\mathcal{M}^{\alpha,\sigma,\varphi}(J_a^b)$ for the full M\"{o}bius group, and under what conditions the series for $R_T$ can be controlled, are open problems left for future work.
\end{remark}

\section*{Statements and Declarations}
\subsection*{Funding} This work was partially supported by Instituto Polit\'ecnico Nacional (grant numbers IND-2026-0491, IND-2026-0101) and SECIHTI (grant number 1179388).
\subsection*{Competing Interests} No competing interests appear to influence the work reported in this paper are disclosed by the authors
\subsection*{Author contributions} IPB and JOGC conceived the study and writing of the paper. JBR and HMN oversaw the project progress. All authors provided critical feedback and helped shape the research, analysis and manuscript. 
\subsection*{ORCID}
\noindent
Isidro Paulino-Basurto: https://orcid.org/0009-0003-4550-4851
\\
Jos\'e Oscar Gonz\'alez-Cervantes: https://orcid.org/0000-0003-4835-5436
\\
Juan Bory-Reyes: https://orcid.org/0000-0002-7004-1794
\\
Hung Manh Nguyen: https://orcid.org/0009-0004-6503-6693


\begin{thebibliography}{99}

\bibitem{GB-1}  Bory-Reyes. J., González-Cervantes, J. O. \textit{A quaternionic proportional fractional  Fueter-type operator calculus}. Bol. Soc. Mat. Mex. (2025) 31:137.

\bibitem{GB-2} Gonz\'alez Cervantes, J. O., Bory-Reyes, J. \textit{A quaternionic fractional Borel-Pompeiu type formula}, Fractal, Vol. 30, No. 1 (2022) 2250013 (15 pages).

\bibitem{GPBN}   González-Cervantes, J. O., Paulino-Basurto, I. , Bory-Reyes, J. , Nguyen. H. M.  \textit{The Borel-Pompieu formula involving proportional fractional $\psi$-Cauchy-Riemann operators}. Adv. Appl. Clifford Algebras (2026) 36:1.

\bibitem{GN1} G\"urlebeck K., Nguyen H.M. \textit{On $\psi$-hyperholomorphic functions in $\mathbb R^3$}. AIP Conference Proceedings. 1558, 496-501, 2013.

\bibitem{GN2} G\"urlebeck K., Nguyen H.M. \textit{On $\psi$-hyperholomorphic functions and a decomposition of harmonics}. Trends in Mathematics/Hypercomplex Analysis: New perspectives and applications, ed. Birkhauser, Basel, 181-189, 2014.

\bibitem{GNL} G\"{u}rlebeck K., Nguyen H. M., Legatiuk D. \textit{$\psi$-Hyperholomorphic functions and a {K}olosov--{M}uskhelishvili formula}. Math.\ Methods Appl.\ Sci., 38(18):5114--5123, 2015.

\bibitem{JAAl} Jarad, F., Abdeljawad, T., Alzabut, J. \textit{Generalized fractional derivatives generated by a class of local proportional derivatives}. Eur. Phys. J. Spec. Top. 226, 3457-3471, 2017.

\bibitem{JAA} Jarad, F., Alqudah, M. A., Abdeljawad, T. \textit{On more generalized form of proportional fractional operators}. Open Mathematics, vol. 18, no. 1, 167-176, 2020.

\bibitem{JUAB} Jarad, F., Ugurlu, E., Abdeljawad, T., Baleanu, D. \textit{On a new class of fractional operators}. Adv. Difference Equ. Paper No. 247, 2017, 16 pp.

\bibitem{LN} Legatiuk, D., Nguyen, H. M. \textit{Improved convergence results for the finite element method with holomorphic functions}.
Adv. Appl. Clifford Algebras, 24(4):1077--1092, 2014.

\bibitem{LWP} Legatiuk, D., Weisz-Patrault, D. \textit{Coupling of complex function theory and finite element method for crack propagation through energetic formulation: conformal mapping approach and reduction to a {R}iemann--{H}ilbert problem}. Comput. Methods Funct. Theory, 22:535--557, 2022.

\bibitem{SKM} Samko, S.~G., Kilbas, A.~A., Marichev, O.~I. \textit{Fractional Integrals and Derivatives: Theory and Applications}. Gordon and Breach Science Publishers, Yverdon, 1993.

\bibitem{S} Shapiro, M. \textit{Quaternionic analysis and some conventional theories}. In: Alpay, D. (ed.) Operator Theory, 1423-1446. Springer, Basel, 2015.

\bibitem{SV1} Shapiro, M, Vasilevski, N. L. \textit{Quaternionic $\psi$-monogenic functions, singular operators and boundary value problems. I. $\psi$-Hyperholomorphy function theory.}  Compl. Var. Theory Appl. 27, 17-46, 1995.

\bibitem{SV2} Shapiro, M., Vasilevski, N. L. \textit{Quaternionic $\psi$-hyperholomorphic functions, singular operators  and boundary value problems II. Algebras of singular integral operators and Riemann type boundary value problems.} Compl. Var. Theory Appl. 27, 67-96, 1995.
\end{thebibliography}
\end{document}